\documentclass{amsart}

\usepackage{amssymb,amsmath}
\usepackage[section]{algorithm}
\usepackage{algpseudocode}

\newtheorem{theorem}{Theorem}[section]
\newtheorem{proposition}[theorem]{Proposition}
\newtheorem{lemma}[theorem]{Lemma}
\newtheorem{corollary}[theorem]{Corollary}

\newtheorem{definition}[theorem]{Definition}

\newtheorem{remark}[theorem]{Remark}

\newtheorem{example}[theorem]{Example}
\newenvironment{Example}{\begin{example}\em}{\end{example}}

\def\KK{{\mathbb K}}

\def\st{\hbox{s.t.}}

\def\Char{{\rm{char}}}

\def\Ker{{\mathrm{Ker\,}}}
\def\Im{{\mathrm{Im\,}}}

\def\Gr{Gr\"obner}

\def\lm{{\mathrm{lm}}}

\def\LM{{\mathrm{LM}}}

\def\HS{{\mathrm{HS}}}
\def\Tor{{\mathrm{Tor}}}

\newcommand{\singular}{{\sc Singular}\ }
\newcommand{\bergman}{{\sc Bergman}\ }

\def\sto{\ {\overset{*}{\to}\ }}
\def\lsto{\ {\overset{*}{\to}_L\, }}
\def\gd{{\mathrm{gl.dim}}}

\begin{document}

\title[Noncommutative algebras, context-free grammars and $\ldots$]
{Noncommutative algebras, context-free grammars and algebraic Hilbert series}

\author[R. La Scala]{Roberto La Scala$^*$}

\author[D. Piontkovski]{Dmitri Piontkovski$^{**}$}

\author[S.K. Tiwari]{Sharwan K. Tiwari$^{***}$}

\address{$^*$ Dipartimento di Matematica,
Universit\`a di Bari, Via Orabona 4, 70125 Bari, Italy}
\email{roberto.lascala@uniba.it}

\address{$^{**}$ Department of Mathematics for Economics,
National Research University Higher School of Economics,
Myasnitskaya str. 20, Moscow 101990, Russia}
\email{dpiontkovski@hse.ru}

\address{$^{***}$ Scientific Analysis Group, Defence Research
\& Development Organization, Metcalfe House, Delhi-110054, India}
\email{shrawant@gmail.com}

\thanks{We acknowledge the support of the University of Bari and
of the visiting program of Istituto Nazionale di Alta Matematica. The research
of the author D.P. is supported by RFBR project 18-01-00908.}


\maketitle

\begin{abstract}
In this paper we introduce a class of noncommutative (finitely generated)
monomial algebras whose Hilbert series are algebraic functions.
We use the concept of graded homology and the theory of unambiguous
context-free grammars for this purpose. We also provide examples
of finitely presented graded algebras whose corresponding leading
monomial algebras belong to the proposed class and hence possess
algebraic Hilbert series.
\end{abstract}

\section{Introduction}

There are few tools in mathematics which have the same applicability,
ubiquity and beauty as the generating functions. Their study is essential,
for instance, for algebra, combinatorics and theoretical computer science.
Indeed, it is generally easier to determine the generating function
of a numerical sequence than its closed formula and by means of such
function one can obtain important data as the asymptotics of the sequence.

Undoubtedly, the most useful generating functions in Algebra are
the Hilbert series (or growth series) of graded and filtered structures.
Among their applications, we mention the possibility to bound Krull
dimensions (via GK-dimensions) and homological dimensions, as well as
their use to characterize the existence of Polynomial Identities,
Noetherianity, Koszulness and other remarkable properties of algebras
(see, for instance \cite{KL,Uf2}). It was Hilbert himself who proved that
(finitely generated) commutative algebras have always rational series,
but the study of Hilbert series of noncommutative structures was started
much later and their behaviour was proved to be wild.

A first result in 1972 is due to Govorov \cite{Go1} who proved that
finitely presented monomial algebras have rational Hilbert series.
He also made the conjecture that all finitely presented graded algebras
have rational series, but a couple of counterexamples were found
by Shearer \cite{Sh} in 1980 and Kobayashi \cite{Ko} in 1981.
We remark that the corresponding non-rational Hilbert series were algebraic
functions, that is, roots of polynomials with coefficients in the rational
function field. At the same time, classes of universal enveloping algebras
with trascendental Hilbert series were also discovered (see, for instance,
\cite{Uf2}). Note also that finitely presented algebras whose
growth is intermediate have necessarily transcendental series. Examples
of such algebras have been recently introduced by Ko\c{c}ak \cite{Kc1,Kc2}.

Another important class of algebras having rational Hilbert series
is the class of {\em automaton algebras} which was introduced by Ufnarovski
in \cite{Uf1} (see also \cite{Ma}). In the theory of formal languages
of theoretical computer science, the regular languages are the ones
that are recognized by finite-state automata and it is well-known that
such languages have rational generating functions (see, for instance, \cite{KS}).
Accordingly, the set of normal words of an automaton algebra is by definition
a regular language. One proves easily that automaton monomial algebras
include finitely presented ones and hence the Govorov's rationality
theorem can be seen as a consequence of a general result from the theory
of formal languages.

Recently, in the papers \cite{LS2,LST} we have proposed algorithms 
for the the computation of the rational Hilbert series of an automaton
algebra which are also implemented in \singular \cite{DGPS}.
These procedures are based on the iteration of a colon (right) ideal
operation and hence they generalize previous methods (see, for instance,
\cite{GP}) for the computation of Hilbert series of commutative algebras.
At the same time, these algorithms for automaton algebras can be
viewed as an application of the Myhill-Nerode theorem for regular languages
(see \cite{DLV}) in the context of noncommutative algebras.

Another fundamental approach to the computation of Hilbert series
consists in determining such series by means of the homology
of a graded algebra. In the commutative case, this essentially
corresponds to obtain the Hilbert series via the computation
of syzygies which was, in fact, the original method of Hilbert.
In the noncommutative case, the homology of monomial algebras
was initially studied by Backelin \cite{Ba1} and later extended
by Anick \cite{An} to general (associative) algebras. An implementation
of effective methods to compute these combinatorial structures can
be found in \cite{Ba2,CPU}. For alternative methods to compute graded
homology see also \cite{LS1}.
The minimal generating sets defining the homology of a monomial algebra
are parametrized by (possibly infinite) sets of words which are called ``chains''.
Therefore, a monomial algebra is associated to a sequence of chain languages
in a natural way and the corresponding Hilbert series can be obtained
by the generating functions of such languages whenever these functions
can be computed.

We introduce homology in Section \ref{homsec} and in Section \ref{govsec}
we prove that automaton algebras have regular chain languages by using
a result of Govorov \cite{Go2}. Therefore, the rational Hilbert series
of such algebras can also be obtained by the generating functions of
their chain languages which are computed by solving a linear system
over the rational function field according to Myhill-Nerode theorem.

With the aim of generalizing this result to algebraic Hilbert series,
we introduce in Section \ref{CFsec} and Section \ref{CSsec} some
fundamental concepts of theoretical computer science such as unambiguous
context-free grammars and languages. By means of a key result of
Chomsky-Sch\"utzenberger for such languages (see, for instance, \cite{KS})
it is well-known that a corresponding generating function is a root
of a univariate polynomial obtained by eliminating over a system
of algebraic equations corresponding to the rules of the grammar
which generates the unambiguous language.

In contrast with the automaton case, we see in Example \ref{countex}
that there are monomial algebras whose chain languages are generally not
context-free even if they are defined by an unambiguous context-free set of relations.
For this reason, we introduce in Section \ref{unsec}, the notion of {\em (homologically)
unambiguous algebra} which is a monomial algebra whose chain languages are all
unambiguous. In the case of finite global dimension, we obtain that such algebras
have computable algebraic Hilbert series. For the elimination process one may use
the computation of a (commutative) \Gr\ basis with respect to an elimination monomial
ordering. In the same section, we prove that there are large classes of unambiguous
algebras. We also show that some of them have infinite global dimension
but computable algebraic series.

In Section \ref{monsec}, we illustrate some examples of unambiguous
monomial algebras with detailed computations of their algebraic Hilbert
series. In Section \ref{fpsec}, we also give some examples of finitely presented
graded algebras whose corresponding leading monomial algebras are homologically
unambiguous. Then, we compute the algebraic Hilbert series of these algebras.
Our main example is a quadratic monoid algebra such that the minimal set of
relations of the associated monomial algebra is a language related to the Dyck
words of balanced brackets.

Finally, in Section \ref{conc} we conclude and propose some suggestions
for further developments of the theory of non-rational Hilbert series.


\section{Monomial algebras and their homology}
\label{homsec}

Let $\KK$ be any field and let $X = \{x_1,\ldots,x_n\}$ be any finite set.
We denote by $X^*$ the free monoid generated by $X$, that is, the elements
of $X^*$ are words on the alphabet $X$.
Consider $F = \KK\langle X\rangle$ the free associative algebra generated
by $X$, that is, $F$ is the vector space which has $X^*$ as a $\KK$-linear
basis. The elements of $X^*$ are also called {\em noncommutative monomials}
and the elements of $F$ are called {\em noncommutative polynomials}
in the variables $x_i\in X$. If $w = x_{i_1}\cdots x_{i_d}\in X^*$,
we put $|w| = d$, that is, $|w|$ is the {\em length of the word $w$}.
The {\em standard grading} $F = \bigoplus_{d\geq 0} F_d$ of the algebra $F$
is hence obtained by defining $F_d\subset F$ as the subspace generated by the set
$\{w\in X^*\mid |w| = d\}$. Any finitely generated algebra with $n$ generators
is clearly isomorphic to a quotient algebra $F/I$ where $I$ is a two-sided
ideal of $F$. The ideal $I$ is called a {\em monomial ideal} if $I$ is generated
by a subset $L\subset X^*$. In this case, we say that $L$ is a {\em monomial
basis of $I$}. A monomial basis $L\subset I$ is called {\em minimal}
if $L$ is an antichain of $X^*$, that is, $v$ is not a subword of $w$,
for all $v,w\in L, v\neq w$. It is well-known that minimal monomial bases
are uniquely defined for all monomial ideals. Owing to non-Noetherianity
of the free associative algebra $F$, note that such bases are generally
infinite sets. The (finitely generated) algebra $A = F/I$ is called a
{\em monomial algebra} if $I$ is a monomial ideal. To simplify notation
for the quotient algebra $A$, we identify words with their cosets
and hence a $\KK$-linear basis of $A$ is given by the {\em set of normal words}
$N(I) = X^*\setminus I$. Let $L_1\subset X^*$ be the minimal monomial basis
of $I$ and assume that $L_1\subset (X^+)^2$ where $X^+ = \{w\in X^*\mid |w|\geq 1\}$.
Then $L_0 = X$ is a minimal generating set of the monomial algebra $A$.
In other words, the sets $L_0,L_1$ define a minimal presentation $A =
\langle L_0\mid L_1 \rangle$. The elements of $L_0$ and $L_1$ are respectively
called the {\em 0-chains} and {\em 1-chains of the monomial algebra $A$}.
According to theoretical computer science, any subset $L\subset X^*$
is called a {\em (formal) language}. Then, the subsets $L_0,L_1\subset X^*$
are languages which are uniquely associated to a monomial algebra $A$.
We show now that there are other languages which correspond uniquely to $A$.
Precisely, we consider the graded homology of a monomial algebra.

Let $A = F/I$ be a monomial algebra such that $L_1\subset (X^+)^2$ is
the 1-chain language of $A$. We consider the subspace $\KK L_0\subset F$
that is generated by $L_0 = X$ and the (vector space) tensor product
$\KK L_0\otimes A$ over the field $\KK$. We can clearly endow $\KK L_0\otimes A$
with the structure of right $A$-module and in fact this is a free module with
basis $\{x_i\otimes 1\mid x_i\in L_0\}$. We can consider hence the right
$A$-module homomorphism
\[
\varphi_0:\KK L_0\otimes A\to A, x_i\otimes 1\mapsto x_i.
\]
Note that $\varphi_0$ is a homogeneous map with respect to the standard grading
of the tensor product $\KK L_0\otimes A$, that is, $|x_i\otimes w| = |x_i| + |w| =
1 + |w|$, for all $x_i\in L_0$ and $w\in N(I)$. The image of this map is clearly
the right (in fact two-sided) ideal $\Im\varphi_0 = \langle x_1,\ldots,x_n \rangle
\subset A$ and therefore we have the graded (right $A$-module) exact sequence
\[
L_0\otimes A\to A\to \KK\to 0.
\]
The {\em graded augmentation map} $A\to \KK$ is clearly defined by composing
the canonical surjection $A\to A/\Im\varphi_0$ with the canonical
isomorphism $A/\langle x_1,\ldots,x_n \rangle\to \KK$.
Recall that the variables $x_i\in L_0$ are called the 0-chains of $A$.
The minimal basis for the (monomial) right $A$-module $\Ker\varphi_0$ is immediately
obtained from the 1-chain language $L_1$ as the set
\[
\bar{L}_1 = \{x_i\otimes t\mid w = x_i t\in L_1, x_i\in L_0, t\in N(I)\}.
\]
Consider the free right $A$-module $\KK L_1\otimes A$ with basis
$\{w\otimes 1\mid w\in L_1\}$. By means of the right $A$-module homomorphism
\[
\varphi_1:\KK L_1\otimes A\to \KK L_0\otimes A, w\otimes 1\mapsto x_i\otimes t
\]
one obtains therefore the exact sequence
\[
\KK L_1\otimes A\to \KK L_0\otimes A\to A\to \KK\to 0.
\]
Observe that $\varphi_1$ is also a homogeneous map and hence we have a graded
exact sequence. This sequence can be further extended in the following way.

Let $w,w'\in L_1$ (possibly $w = w'$) such that $w u = v w'$ with
$u,v\in X^*$ and $|u| < |w'|, |v|\geq 1$. If $w = x_i t$ where $x_i\in L_0, t\in N(I)$,
we have that $w'$ is a subword of $t u$, that is, $t u\in I$ and therefore
\[
\varphi_1(w\otimes u) = x_i\otimes t u = 0.
\]
In other words, the element $w\otimes u\in \KK L_1\otimes A$ belongs to
$\Ker\varphi_1$. One proves easily that
\[
\bar{L}'_2 = \{w\otimes u\mid w u = v w', w,w'\in L_1,
u,v\in X^*, |u| < |w'|, |v|\geq 1\}
\]
is a basis of the (monomial) right submodule $\Ker\varphi_1\subset \KK L_1\otimes A$.
Observe that $\bar{L}'_2$ is generally not the minimal basis of $\Ker\varphi_1$
because we may have two elements $w\otimes u, w\otimes u'$ such that
$u v = u'$, for some $v\in X^*$. Then, consider the minimal basis $\bar{L}_2$
obtained from $\bar{L}'_2$ by discarding redundant generators and define
the language
\[
L_2 = \{w u\mid w\otimes u\in \bar{L}_2\}.
\]
Clearly, we may define the homogeneous right $A$-module homomorphism
\[
\varphi_2:\KK L_2\otimes A\to \KK L_1\otimes A,
w u\otimes 1\mapsto w\otimes u
\]
and hence the graded exact sequence
\[
\KK L_2\otimes A\to \KK L_1\otimes A\to \KK L_0\otimes A\to A\to \KK\to 0.
\]
The elements of $L_2$ are called the {\em 2-chains of $A$}. One has possibly
$L_2 = \emptyset$ and hence $\KK L_2\otimes A = 0$. In this case,
the minimal monomial basis $L_1$ is called {\em combinatorially free}.

If $L_2\neq \emptyset$, we can continue in this process of defining a minimal
(graded) free right $A$-module resolution of the base field $\KK$
\begin{equation}
\label{resol}
\ldots \to \KK L_{i+1}\otimes A\to \KK L_i\otimes A\to \ldots
\to \KK L_0\otimes A\to A\to \KK\to 0.
\end{equation}
The general map $\varphi_{i+1}:\KK L_{i+1}\otimes A\to \KK L_i\otimes A$
of this resolution can be defined in the following way. We look for
a minimal monomial basis of the right submodule $\Ker\varphi_i\subset
\KK L_i\otimes A$. Let $w = s t\in L_i$ where $s\in L_{i-1}$ and $t\in N(I)$,
that is, $\varphi_i(w) = s\otimes t$. We call $s$ the {\em prefix $(i-1)$-chain}
and $t$ the {\em tail of the $i$-chain $w$}. Let $w'\in L_1$ and assume that
$w u = v w'$ with $u,v\in X^*$ and $|u| < |w'|, |v|\geq |s|$. We have clearly
that $w'$ is a subword of $t u$, that is, $t u\in I$ and hence
\[
\varphi_i(w\otimes u) = s\otimes t u = 0.
\]
In other words, the element $w\otimes u\in \KK L_i\otimes A$ belongs to
$\Ker\varphi_i$. Indeed, one has that
\[
\bar{L}'_{i+1} = \{w\otimes u\mid w u = v w', w = s t\in L_i, w'\in L_1,
u,v\in X^*, |u| < |w'|, |v|\geq |s|\}
\]
is a basis of the (monomial) right submodule $\Ker\varphi_i\subset \KK L_i\otimes A$.
We consider then the minimal basis $\bar{L}_{i+1}$ obtained from $\bar{L}'_{i+1}$
and we define the language
\[
L_{i+1} = \{w u\mid w\otimes u\in \bar{L}_{i+1}\}.
\]
We have obtained therefore the required homogeneous right $A$-module homomorphism
\[
\varphi_{i+1}:\KK L_{i+1}\otimes A\to \KK L_i\otimes A,
w u\otimes 1\mapsto w\otimes u.
\]
Observe that the elements $w u$ such that $w\otimes u\in \bar{L}'_{i+1}$
are sometimes called {\em $(i+1)$-prechains of $A$}. The set $L_{i+1}$ is called
the {\em $(i+1)$-chain language of $A$}. Since $\bar{L}_{i+1}$ is the unique
minimal monomial basis of $\Ker\varphi_i$, we have that $L_{i+1}$ is uniquely
associated to $A$, for all $i$. Then, we introduce the notation $L_i(A)$
for the $i$-chain language of the monomial algebra $A$.

The resolution (\ref{resol}) is called the {\em (graded) homology of the monomial
algebra $A$}. In fact, by putting $L_{-1} = 1$, that is, $\KK L_{-1} = \KK$,
we have that $\Tor_i^A(\KK,\KK)$ and $\KK L_{i-1}$ are isomorphic graded vector
spaces, for $i\geq 0$. For a reference to the concept of $\Tor$ functor,
we refer to \cite{Uf2}.  In general, the homology has not a finite length.
If, instead, one has a finite resolution
\[
0\to \KK L_k\otimes A\to\ldots\to \KK L_1\otimes A\to \KK L_0\otimes A
\to A \to \KK\to 0
\]
we call the integer $\gd(A) = k + 1$ the {\em global (homological) dimension of $A$}.
Both for finite or infinite global dimension, observe that if $L_1$ is a finite
language then all other languages $L_i$ are also finite. We will see
in Section \ref{govsec} that this property can also be extended to some good
class of infinite languages.


\section{Context-free grammars}
\label{CFsec}

A fundamental way to define infinite languages is via recursion. 
In theoretical computer science, this is formalized by the notion 
of (formal) grammar. Since in this paper we are mainly interested
in computing algebraic generating functions of languages, we restrict
ourselves to consider context-free grammars (see Theorem \ref{CSth}).
Consider \cite{DLV,HMU,KS} as general references for the results contained
in this section.

\begin{definition}
A {\em context-free grammar}, briefly a {\em cf-grammar}, is a quadruple
$G = (V, X, P, S)$, where $V, X$ are finite sets, $P$ is a finite subset
of $V\times (V\cup X)^*$ and $S\in V$. An element $(A,\alpha)\in P$ is usually
denoted as $A\to \alpha$ where $A\in V$ and $\alpha$ is a word on the
alphabet $V\cup X$. An element $A\to \alpha$ is called a {\em production}
or a {\em rule} of the grammar $G$. The elements of $V$ are called
{\em variables} or {\em nonterminals} and the elements of $X$ are called
{\em terminals}. The distinguished variable $S\in V$ is called the
{\em start variable}.  If there are different productions
$A\to\alpha_1,\ldots,A\to\alpha_k$ for the same variable $A$,
one uses the compact notation $A\to \alpha_1\mid\ldots\mid\alpha_k$.
\end{definition}

For the convenience of readers, who may not be familiar with theoretical
computer science, we provide immediately an example to let them enter into
the idea behind the formalism. As in the previous section,
let $X = \{x_1,\ldots,x_n\}$. If $w = x_{i_1}\cdots x_{i_d}\in X^*$ ($d\geq 0$),
we denote $w^R = x_{i_d}\cdots x_{i_1}$. A word such that $w = w^R$
is called a {\em palindrome}. Then, consider the language of all palindromes
$L = \{w\in X^*\mid w = w^R\}$. It is clear that $1, x_i\in L$ and
$x_i L x_i\subset L$, for any $x_i\in L$. In fact, the language $L$
has the following recursive structure
\[
L = \{1\}\cup\{x_1\}\cup\ldots\cup\{x_n\}\cup x_1 L x_1\cup\ldots\cup x_n L x_n.
\]
Then, if $V = \{S\}$ and
\[
P = \{S\to 1\mid x_1\mid\ldots\mid x_n\mid x_1 S x_1\mid\ldots\mid x_n S x_n\},
\]
the language $L$ may be represented as the cf-grammar $G = (V, X, P, S)$.
Actually, this is a very simple example where recursion involves a single language
and hence the grammar needs only the start variable. In general, recursion may be
more involved relating different languages, that is, $V$ may contain many variables.

Note that beside the understanding of a cf-grammar as a recursion process
defining a language, there exists another viewpoint which is predominant
in computer science. This consists in considering a production
$A\to \alpha$ as a rewriting rule where the variable $A$ is rewritten as
the word $\alpha$. Starting with the variable $S$, one iterates this
rewriting process for all variables (nonterminals) occurring in the current
word untill we obtain a word whose letters are all terminals. This way
to obtain the language from the cf-grammar is called ``derivation''.
Consider, for instance, the language of palindromes
$L = \{w\in X^*\mid w = w^R\}$. As an example of derivation, we can rewrite
$S$ by the word $x_1Sx_1$ ($S\to x_1Sx_1$) which can be rewritten to
$x_1x_2Sx_2x_1$ ($S\to x_2Sx_2$) and finally to $x_1x_2x_1x_2x_1$ ($S\to x_1$)
which is a palindrome. One can obtain all words of the language $L$
by means of such derivations.

\begin{definition}
Let $G = (V, X, P, S)$ be a cf-grammar and let $v,w\in (V\cup X)^*$.
We denote $v\to w$ if $v = p A q, w = p \alpha q$, where $A\to\alpha\in P$
and $p,q\in (V\cup X)^*$. A {\em derivation from $v$ to $w$ of length $k$}
$(k\geq 1)$ is a sequence $(v_0,v_1,\ldots,v_k)$ of words of $(V\cup X)^*$
such that $v_{i-1}\to v_i$ and $v = v_0, w = v_k$. In this case, we write
$v\overset{k}{\to} w$. Moreover, we put $v\overset{0}{\to} v$.
We denote $v\sto w$ if $v\overset{k}{\to} w$, for some $k\geq 0$.
The {\em cf-language generated by $G$} is by definition
\[
L(G) = \{w\in X^*\mid S\sto w\}.
\]
\end{definition}

Among derivations, there are some canonical ones that are sufficient
to generate a cf-language.

\begin{definition}
Let $v\sto w$ be a derivation, that is, $v = v_0\to v_1\to \ldots \to v_k = w$.
If, for any derivation step $v_i\to v_{i+1}$, the variable in the word $v_i$
that is rewritten by a production is exactly the leftmost occurrence
of a variable in $v_i$, then we call $v\sto w$ a {\em leftmost derivation}
and we denote it as $v\lsto w$.
\end{definition}

To provide an example of a leftmost derivation, consider the Dyck grammar
$G = (V,X,P,S)$ where $V = \{S\}, X = \{x,y\}, P = \{S\to 1\mid x S y S\}$.
If one understands the terminals $x,y$ as a left and a right bracket,
we have that $L(G)$ is the language of words of balanced brackets.
One has the following leftmost derivation
\[
S\to x S y S \to x y S \to x y x S y S \to x y x y S \to x y x y.
\]

In Section \ref{CSsec}, we will see that a fundamental property which allows to count
the words of a cf-language is the following one.

\begin{definition}
A cf-grammar $G$ is called {\em unambiguous} if for any word $w\in L(G)$, there is
a unique leftmost derivation $S\lsto w$. In other words, one has a bijection
between the cf-language $L(G)$ and the set of leftmost derivations
$\{S\lsto w\mid w\in L(G)\}$.
\end{definition}

Note that different cf-grammars may define the same cf-language. So, a cf-language
may be defined both from an ambiguous and an unambiguous cf-grammar. In fact,
there are cf-languages where all corresponding cf-grammars are ambiguous.
These cf-languages are called {\em inherently ambiguous}. All other cf-languages
are called {\em unambiguous}.

In the class of unambiguous cf-languages we find the regular languages that
can be defined in terms of cf-grammars in the following way.

\begin{definition}
\label{rgram}
A cf-grammar $G = (V, X, P, S)$ is called {\em regular} or {\em right linear}
if all productions are of type $A\to 1$ or $A\to x_i B$, where $A,B\in V,
x_i\in X$. The corresponding cf-language $L(G)$ is called a {\em regular
language}. 
\end{definition}

By the right linearity of productions, regular languages are clearly unambiguous
ones. Moreover, finite languages are regular ones. The regular languages can be
characterized in different ways. An important characterization is provided
by the Myhill-Nerode theorem. To state this theorem, we have to introduce
the following notions.

\begin{definition}
Let $X = \{x_1,\ldots,x_n\}$ and consider $L\subset X^*$ and $w\in X^*$.
We define the {\em right quotient of the language $L$ by the word $w$}
as the language
\[
w^{-1} L = \{v\in X^*\mid w v\in L\}.
\]
Moreover, we put $Q(L) = \{w^{-1} L\mid w\in X^*\}$.
\end{definition}

\begin{theorem}[Myhill-Nerode]
The language $L$ is regular if and only if $Q(L)$ is a finite set.
\end{theorem}

Note that $w^{-1} v^{-1} L = (v w)^{-1} L$, for all $v,w\in X^*$
and $1^{-1} L = L$. Then, $Q(L)$ is the smallest set of languages
containing $L$ such that $x_i^{-1} L'\in O(L)$, for all $L'\in Q(L)$
and for any $x_i\in X$. Moreover, one has clearly that
$L = x_1^{-1} L \cup\ldots\cup x_n^{-1} L\cup C$, where
\[
C =
\left\{
\begin{array}{cl}
\{1\} & \mbox{if}\ 1\in L, \\
\emptyset & \mbox{otherwise}.
\end{array}
\right.
\]
The above theorem provides hence a procedure to construct a regular grammar
generating a given regular language. Such grammar is minimal with respect
to the number of productions.

\suppressfloats[b]
\floatname{algorithm}{Algorithm}
\begin{algorithm}\caption{}
\label{orbit}
\begin{algorithmic}[0]
\State \text{Input:} $L$, a regular language.
\State \text{Output:} $G$, a (minimal) regular grammar \st\ $L(G) = L$.
\State $N:= \{L\}$;
\State $Q:= \{L\}$;
\State $V:= \{A_1\}$;
\While{$N\neq\emptyset$}
\State {{\bf choose}} $L'\in N$;
\State $N:= N\setminus\{L'\}$;
\State $k:= $ the position of $L'$ in $Q$;
\If{$1\in L'$}
\State $P:= P\cup\{A_k\to 1\}$;
\EndIf;
\ForAll{$1\leq i\leq n$}
\State $L'':= x_i^{-1} L'$;
\If{$L''\notin Q$}
\State $N:= N\cup\{L''\}$;
\State $Q:= Q\cup\{L''\}$;
\State $V:= V\cup\{A_{\#Q}\}$;
\EndIf;
\State $l:= $ the position of $L''$ in $Q$;
\State $P:= P\cup\{A_k\to x_i A_l\}$;
\EndFor;
\EndWhile;
\State \Return $G = (V,X,P,A_1)$.
\end{algorithmic}
\end{algorithm}

Note that in the above procedure the index $\#Q$ of the variable $A_{\#Q}$
denotes the cardinality of the current set $Q$.

As an example of application of this procedure, let $X = \{x,y\}$ and
consider the regular language $L_1 = \{x^my^n\mid m,n\geq 0\}\subset X^*$.
We have clearly that $1\in L_1$ and $x^{-1} L_1 = L_1, y^{-1} L_1 = L_2$,
where $L_2 = \{y^n\geq 0\}$.
Moreover, we have that $1\in L_2$ and $x^{-1} L_2 = L_3, y^{-1} L_2 = L_2$
where $L_3 = \emptyset$.  Finally, one has that $x^{-1} L_3 = y^{-1} L_3 = L_3$.
We conclude that $Q(L) = \{L_1,L_2,L_3\}$ and the regular grammar
$G = (V, X, P, S)$ corresponding to $L_1$ is defined as
$V = \{A_1,A_2,A_3\}, S = A_1$ and
\[
P = \{A_1\to 1\mid x A_1\mid y A_2,
A_2\to 1\mid x A_3\mid y A_2,
A_3\to x A_3\mid y A_3\}.
\]
Observe that if we apply the above procedure to a non-regular language
$L$, one obtains an infinite set of right quotients $Q(L)$, that is,
an ``infinite regular grammar''.

There are many operations one can consider among languages. The regular languages
are closed with respect to almost all of them but general cf-languages are closed
only for some of them.

\begin{definition}
Given two languages $L,L'\subset X^*$, we consider the set-theoretic union
$L\cup L'$ and the product $L L' = \{w w'\mid w\in L,w'\in L'\}$.
Moreover, one defines the {\em star operation} $L^* = \bigcup_{d\geq 0} L^d$
where $L^0 = \{1\}$ and $L^d = L L^{d-1}$, for any $d\geq 1$. The union,
the product and the star operation are called the {\em regular operations
over the languages}. One also considers the set-theoretic intersection
$L\cap L'$ and the complement $L^c = \{w\in X^*\mid w\notin L\}$.
\end{definition}

\begin{proposition}
\label{closure}
The regular languages and cf-languages are closed under the regular operations.
The regular languages are also closed under intersection and complement.
Moreover, if $L$ is a cf-language and $L'$ is a regular one then
$L\cap L'$ is a cf-language.
\end{proposition}

In fact, regular languages can be obtained from finite languages by means of
regular operations.

\begin{theorem}[Kleene]
A language $L\subset X^*$ is {\em regular} if and only if it can be obtained
from finite languages by applying a finite number of regular operations.
\end{theorem}

For instance, the regular language $L = \{x^my^n\mid m,n\geq 0\}$ can be
obtained as $L_1^* L_2^*$ where $L_1 = \{x\}$ and $L_2 = \{y\}$.
Indeed, we may use a {\em regular expression} to denote any element of $L$,
namely $x^*y^*$.


\section{Hilbert series and generating functions}
\label{CSsec}

Let $V\subset F$ be a graded subspace of $F$, that is, $V = \sum_{d\geq 0} V_d$
where $V_d = V\cap F_d$. The {\em Hilbert series of the graded subspace $V$}
is by definition the generating function of the sequence $\{\dim V_d\}_{d\geq 0}$,
namely
\[
\HS(V) = \sum_{d\geq 0} (\dim V_d) t^d.
\] 
In theoretical computer science, a similar notion is provided for any language
$L\subset X^*$. For all $d\geq 0$, put $L_d = \{w\in L\mid |w| = d\}$
and denote by $\# L_d$ the cardinality of the set $L_d$.
The {\em generating function of the language $L$} is defined as
the generating function of the sequence $\{\# L_d\}_{d\geq 0}$, that is
\[
\gamma(L) = \sum_{d\geq 0} (\# L_d) t^d = \sum_{w\in L} t^{|w|}.
\]
If $V = \KK L$ is the (graded) subspace of $F$ that is generated by $L$,
one has clearly that $\HS(V) = \gamma(L)$. For instance, we have that
$\HS(F) = \gamma(X^*) = 1/(1 - nt)$. If $A = F/I$ is a monomial algebra,
we define its {\em Hilbert series} as $\HS(A) = \HS(F) - \HS(I)$.
By denoting $L(I) = I\cap X^*$, we have that $I = \KK L(I)$ and hence
$\HS(I) = \gamma(L(I))$. Observe that if $L_1\subset X^*$ is the minimal
monomial basis of $I$ then $L(I) = X^* L_1 X^*$. We call $L(I)$
the {\em language of the monomial ideal $I$}.

For noncommutative algebras, the sum of the series $\HS(A)$ is either a rational
or a non-rational function. The computation of this fundamental invariant
is the main goal of the present paper. Besides the rational case,
we are especially interested when the Hilbert series $\HS(A) = f(t)$ is an
algebraic function, that is, $f(t)$ is an algebraic element over the rational
function field $\KK(t)$. In other words, $f(t)$ is a root of a commutative
univariate polynomial with coefficients in $\KK(t)$, or equivalently
in $\KK[t]$. In this case, the task clearly becomes to compute such a
polynomial.

If $A$ has a finite global dimension, the homology of $A$ provides
immediately a way to compute $\HS(A)$. In fact, since we have a graded
exact sequence
\[
0\to \KK L_k\otimes A\to\ldots\to \KK L_1\otimes A\to \KK L_0\otimes A
\to A \to \KK\to 0
\]
one has immediately that
\[
\sum_{i = 0}^k (-1)^i \gamma(L_i)\HS(A) - \HS(A) + 1 = 0.
\]
Because $L_0 = X$, we finally obtain the formula
\begin{equation}
\label{hilbhom}
\HS(A) = \left( 1 - n t - \sum_{i=1}^k (-1)^i \gamma(L_i) \right)^{-1}.
\end{equation}
We conclude that the Hilbert series of a monomial algebra $A$ is directly
related to the generating functions of the chain languages $L_i = L_i(A)$.
If the 1-chain language $L_1$ is a finite language, we have already observed
that all $L_i$ are also finite languages, that is, $\gamma(L_i)$ are polynomials.
In this case, therefore, the Hilbert series $\HS(A)$ is a rational function
which can be immediately computed using the formula (\ref{hilbhom}). When
the $L_i$ are instead infinite sets, it is important to understand if and
how one can compute their generating functions. Theoretical computer science
provides a positive answer for the class of unambiguous cf-languages.

\begin{definition}
\label{system}
Let $G = (V,X,P,S)$ be an unambiguous cf-grammar where
$V = \{A_1,\ldots,A_m\}, S = A_1, X = \{x_1,\ldots,x_n\}$ and
$P = \{A_i \to \alpha_{i1}\mid \ldots\mid \alpha_{ik_i}\}_{1\leq i\leq m}$
where $\alpha_{ij}\in (V\cup X)^*$. Let $t\notin V\cup X$ be a new
variable and consider the rational function field $\KK(t)$. Moreover,
consider the algebra $R = \KK(t)[A_1,\dots,A_m]$ of the commutative
polynomials in the variables $A_i\in V$ with coefficients in the field $\KK(t)$.
The product of a coefficient in $\KK(t)$ with a monomial of $R$
is called a {\em term of $R$}. To each production
$A_i \to \alpha_{i1}\mid \ldots\mid \alpha_{ik_i}$ we can associate
the commutative polynomial $A_i - \bar{\alpha}_{i1} - \cdots - \bar{\alpha}_{ik_i}$
in $R$, where $\bar{\alpha}_{ij}$ is obtained from the word $\alpha_{ij}$
by substituting each terminal $x_i\in X$ with the variable $t$ and
then transforming the resulting word into a term of the algebra $R$.
For instance, if $\alpha = A_2 x_1 A_1 x_2 A_2$ then $\bar{\alpha} =
t^2 A_1 A_2^2$. Then, denote by $S(G)$ the following system of algebraic
equations
\[
S(G):
\left\{
\begin{array}{ccl}
A_1 & = & \bar{\alpha}_{11} + \cdots + \bar{\alpha}_{1k_1}, \\
\vdots \\
A_m & = & \bar{\alpha}_{m1} + \cdots + \bar{\alpha}_{mk_m}.
\end{array}
\right.
\]
We call $S(G)$ the {\em algebraic system of the (unambiguous) cf-grammar $G$}.
\end{definition}

Given an unambiguous cf-grammar $G$, observe that all variables $A_i\in V$
correspond to languages
\[
L_G(A_i) = \{w\in X^*\mid A_i\sto w\}.
\]
One has clearly that $L(G) = L_G(S)$ and all $L_G(A_i)$ are unambiguous
cf-languages. A fundamental result is the following one.

\begin{theorem}[Chomsky-Sch\"utzenberger]
\label{CSth}
Let $G = (V,X,P,S)$ be an unambiguous cf-grammar as in Definition \ref{system}
and denote $\gamma_i = \gamma(L_G(A_i))$, for any $i = 1,2,\ldots,m$. Then,
the $m$-tuple $(\gamma_1,\ldots,\gamma_m)$ is a solution of the algebraic
system $S(G)$. Moreover, each $\gamma_i$ is an algebraic function.
\end{theorem}

We refer to \cite{KS,Pa,St} for proofs of the above result. In these references
one also finds methods for computing a univariate polynomial $0\neq p(A_i)\in \KK(t)[A_i]$
such that $p(\gamma_i) = 0$. Even if these general procedures are quite involved,
one has a simpler method that works for many concrete grammars. We assume
the reader is familiar with the theory of (commutative) \Gr\ bases. For a complete
reference see, for instance, \cite{GP}.

\begin{theorem}
\label{lexgb}
Let $\KK'$ be any field and consider the commutative polynomial algebra
$R = \KK'[A_1,\ldots,A_m]$. Let $I = \langle f_1,\ldots,f_k \rangle$ be an
ideal of $R$ and consider the corresponding algebraic system
\[
S:
\left\{
\begin{array}{ccl}
f_1 & = & 0, \\
\vdots \\
f_k & = & 0.
\end{array}
\right.
\]
Denote $R_1 = \KK'[A_1]$ and put $I_1 = I\cap R_1$ which is an ideal of the univariate
polynomial algebra $R_1$. Let $G = \{g_1,\ldots,g_l\}\subset R$ be a \Gr\ basis of $I$
with respect to the lexicographic monomial ordering of $R$ such that $A_1\prec \ldots \prec A_m$.
If $I_1\neq 0$, there exists $g_i\in G$ such that $g_i\in I_1$. We have therefore that
$g_i(\gamma_1) = 0$, for all solutions $(\gamma_1,\ldots,\gamma_m)$ of the algebraic
system $S$.
\end{theorem}

\begin{Example}
Consider the unambiguous cf-grammar $G = (V,X,P,S)$ where $V = \{S, A, B\},
X = \{x,y\}$ and
\[
P = \{S\to A\mid B, A\to 1\mid x A y A, B\to x S\mid x A y B\}.
\]
This grammar generates all unambiguous expressions in the clauses
``if-then-else'' (matched) and ``if-then'' (unmatched). We denote
``if-then'' with the terminal $x$ and ``else'' with the terminal $y$.
The variable $A$ represents the matched clauses and $B$ the unmatched
ones. For instance, in the language $L(G)$ we have the word $w = x y x x y y$
corresponding to the expression
\begin{center}
\em
\begin{tabular}{l}
if-then \\
else \\
if-then \\
\quad if-then \\
\quad else \\
else. \\
\end{tabular}
\end{center}
Since $G$ is unambiguous, by Theorem \ref{CSth} we know that the generating
functions $\gamma_S = \gamma(L_G(S)), \gamma_A = \gamma(L_G(A))$ and 
$\gamma_B = \gamma(L_G(B))$ are algebraic ones. We show now that we can
compute the corresponding univariate polynomials $p_S,p_A,p_B$ with
coefficients in $\KK(t)$ by means of \Gr\ bases computations.
Consider the ideal $I = \langle f_1,f_2,f_3 \rangle$ of the commutative algebra
$R = \KK(t)[S,A,B]$ that is generated by the polynomials corresponding
to the algebraic system $S(G)$, namely
\[
\begin{array}{ccl}
f_1 & = & S - A - B, \\
f_2 & = & A - 1 - t^2 A^2, \\
f_3 & = & B - t S - t^2 A B. \\
\end{array}
\]
Assume $\Char(\KK) = 0$. With respect to the lexicographic monomial
ordering of the polynomial algebra $R$ with $S\prec A\prec B$,
we obtain the following \Gr\ basis of the ideal $I$
\[
\begin{array}{ccl}
g_1 & = & t(2t - 1) S^2 + (2t - 1) S + 1, \\
g_2 & = & t A - (2t - 1) S - 1, \\
g_3 & = & B + A - S.
\end{array}
\]
This implies that $p_S = g_1$. For $A\prec B\prec S$,
we obtain the \Gr\ basis
\[
\begin{array}{ccl}
g_1 & = & t^2 A^2 - A + 1, \\
g_2 & = & (2t - 1) B - t^2 A^2 + t A, \\
g_3 & = & S - B - A.
\end{array}
\]
Therefore, one has that $p_A = g_1$. Finally, for $B\prec S\prec A$,
we have the \Gr\ basis 
\[
\begin{array}{ccl}
g_1 & = & t^2(2t - 1) B^2 + (t + 1)(2t - 1) B + t, \\
g_2 & = & (t-1) S + t B + 1, \\
g_3 & = & A - S + B.
\end{array}
\]
We conclude that $p_B = g_1$. One of the roots of the
polynomial $p_S$ is
\[
\gamma_S =
- \frac{2t - 1 + \sqrt{1 - 4 t^2}}{2t(2t - 1)}
\]
which admits a power series expansion having the correct (non-negative
integer) coefficients, namely
\[
\gamma_S =
1 + t + 2t^2 + 3t^3 + 6t^4 + 10t^5 + 20t^6 + 35t^7 + \ldots =
\sum_{d = 0}^\infty \binom{d}{\lfloor d/2 \rfloor} t^d.
\]
In the same way, one obtains the algebraic functions $\gamma_A,\gamma_B$.
\end{Example}

We conclude this section by observing that Theorem \ref{CSth} can be
clearly applied to regular, that is, right linear grammars
(see Definition \ref{rgram}), where the corresponding algebraic systems
are in fact linear ones and \Gr\ bases computations are just Gaussian
eliminations over the rational function field $\KK(t)$. In other words,
the generating function of a regular language is a rational function.
These ideas can be applied to monomial algebras by means of the following
notion.

\begin{definition}
Let $A = F/I$ be a monomial algebra and let $L = L(I) = I\cap X^*$
be the language of the monomial ideal $I$. We call $A$ an {\em automaton
(monomial) algebra} when $L$ is a regular language.
\end{definition}

Note that the term ``automaton'' corresponds to the concept of finite-state
automata which are the recognizer machines of regular languages.
By the closure properties of regular languages of Proposition \ref{closure},
one has that $A$ is an automaton algebra if and only if the normal words
language $N(I) = X^*\setminus L(I)$ is a regular one. Moreover, in Section
\ref{govsec} we will prove that the automaton property is also equivalent
to require that the 1-chain language $L_1(A)$ is a regular language.

Observe now that if $A = F/I$ is an automaton algebra then $\HS(A)$
is a rational function which can be computed in the following way.
If $L = L(I)$, we have clearly that $\HS(A) = 1/(1 - nt) - \gamma(L)$.
Since $L$ is a regular language, one has that Algorithm \ref{orbit} computes
a (minimal) regular grammar $G$ generating $L$ and hence the corresponding
linear system $S(G)$. By solving such system, we obtain finally the
rational function $\gamma(L)$ and hence $\HS(A)$. An improved version
of this method can be found in \cite{LS2} and \cite{LST} where an efficient
implementation in \singular has also been provided. The improvement essentially
consists in obtaining the elements of $Q(L)$ by means of computations
over the minimal monomial basis $L_1(A)\subset I$.

A cf-grammar $G$ is called {\em linear} if the right-hand sides of productions
are linear with respect to nonterminals. It is important to note that
a cf-grammar may be linear without being right linear, that is, regular.
In this case, the Myhill-Nerode approach fails to find the rational
generating function of the language $L = L(G)$ because the set $Q(L)$
is infinite, but Chomsky-Sch\"utzenberger's algebraic system $S(G)$
fits for purpose when $G$ is unambiguous. As an example, consider again
the palindromes language $L$ whose unambiguous linear grammar $G$
has the following set of productions 
\[
P = \{S\to 1\mid x_1\mid\ldots\mid x_n\mid x_1 S x_1\mid\ldots\mid x_n S x_n\}.
\]
The corresponding algebraic system $S(G)$ is the single linear equation
\[
S = 1 + n t + n t^2 S
\]
whose rational solution is $\gamma(L) = (1 + n t)/(1 - n t^2)$.


\section{A description of the chain languages}
\label{govsec}

In this section we provide formulas for the chain languages $L_k = L_k(A)$ ($k\geq 1$)
of a monomial algebra $A = F/I$ in terms of the powers of the language $L(I) =
X^* L_1 X^*$. By the closure properties of regular languages, this implies that
if $L_1$ is a regular language then all chain languages $L_k$ are also regular,
that is, the regularity property propagates along the homology of $A$.
We provide, instead, an example of an unambiguous cf-language $L_1$ such that
$L_2$ is not a cf-language. 

If $I,J$ are two-sided ideals of $F$, one defines the two-sided ideal $I J\subset F$
as the subspace generated by the set-theoretic product $\{f g\mid f\in I,g\in J\}$.
In particular, one can define the powers $I^k$ ($k\geq 0$) where by definition
$I^0 = F$. We consider also the (maximal) graded two-sided ideal $F_+ =
\sum_{d > 0} F_d = \langle x_1,\ldots,x_n \rangle$ so that $F/F_+$ is isomorphic
to the base field $\KK$. The following result is a simplified version of Lemma 3
in \cite{Go2}.

\begin{lemma}
\label{govorov}

Let $I\subset F$ be a graded two-sided ideal and consider $A = F/I$
the corresponding (finitely generated) graded algebra. Moreover, assume that
$I\subset F_+$. For all $k\geq 0$, consider the graded two-sided ideals
$J_{2k} = I^k$ and $J_{2k+1} = F_+ I^k$. Since $I^0 = F$, note that $J_0 = F,
J_1 = F_+$. For all $k\geq 1$, one has the following isomorphism of graded
vector spaces
\[
\Tor_k^A(\KK,\KK) \approx (J_k\cap J_{k-1} F_+) / (J_{k+1} + J_k F_+).
\]
\end{lemma}

Recall now that if $A = F/I$ is a monomial algebra where
$L_1 = L_1(A)\subset (X^+)^2 = \{w\in X^*\mid |w|\geq 2\}$, then the homology
of $A$ is obtained in terms of the chain languages $L_k = L_k(A)$, that is,
$\Tor_k^A(\KK,\KK)$ is isomorphic to $\KK L_{k-1}$ ($k\geq 0$) as a graded
vector space. In fact, the algebra $F$ and hence the monomial algebra $A$
are graded by the free monoid $X^*$ and we have that $\Tor_k^A(\KK,\KK)$
and $\KK L_{k-1}$ are isomorphic as $X^*$-graded vector spaces.

Since $X^+ = \{w\in X^*\mid |w|\geq 1\}$, we have clearly that $F_+ = \KK X^+$.
We can restate Lemma \ref{govorov} in terms of the chain languages
in the following way.

\begin{theorem}
\label{govorov2}

Let $L_1\subset (X^+)^2$ be a minimal monomial basis and denote $L = X^* L_1 X^*$.
For all $k\geq 1$, it holds that
\begin{equation}
\label{govform}
\begin{array}{rcl}
L_{2k} & = & (X^+ L^k \cap L^k X^+)\setminus (X^+ L^k X^+ \cup L^{k+1}), \\
L_{2k-1} & = & (X^+ L^{k-1} X^+ \cap L^k) \setminus (X^+ L^k \cup L^k X^+).
\end{array}
\end{equation}
\end{theorem}

\begin{proof}
Let $I = \KK L\subset F$ and consider $A = F/I$ the corresponding monomial algebra.
It is useful to recall some arguments from the proof of Lemma 3 in \cite{Go2}.
The graded vector space isomorphisms of Lemma \ref{govorov} follow from the exact
sequence
\[
\ldots \to J_3/J_5 \to J_2/J_4 \to J_1/J_3 \to A \to \KK \to 0
\]
which is a graded free right $A$-module resolution of the base field $\KK$.
The maps above are induced by inclusions and quotients of ideals of $F$
which are indeed monomial ideals in our case. Therefore, one has a free
resolution in the category of $X^*$-graded right $A$-modules. It follows that
\[
\KK L_{k-1} \approx \Tor_k^A(\KK,\KK) \approx
(J_k\cap J_{k-1} F_+) / (J_{k+1} + J_k F_+)
\]
are in fact isomorphisms of $X^*$-graded vector spaces. Note that $F_w = \KK w$
is the (monodimensional) graded component of $F$ corresponding to a word $w\in X^*$. 
Therefore, we have that the corresponding graded component of $\KK L_{k-1}$
is either $\KK w$ if $w\in L_{k-1}$ or zero otherwise. Since
$(J_k\cap J_{k-1} F_+) / (J_{k+1} + J_k F_+)$ is the quotient of two monomial
ideals of $F$, from the above isomorphisms it follows that a word $w\in X^*$
belongs to $L_{k-1}$ if and only if it belongs to $B_1 \setminus B_2$,
where $B_1$ and $B_2$ are the monomial linear bases of the ideals
$J_k\cap J_{k-1} F_+$ and $J_{k+1} + J_k F_+$, respectively. Because
the monomial linear bases of the ideals $F_+,J_{2k}$ and $J_{2k+1}$ are clearly
the sets $X^+, L^k$ and $X^+ L^k$ respectively, we finally obtain the stated
formulas for the languages of chains.
\end{proof}

For $k = 1,2$, observe that one has the following simplified formulas
\[
\begin{array}{c}
L_1 = ((X^+)^2\cap L)\setminus (X^+ L \cup L X^+) = L \setminus (X^+ L \cup L X^+), \\
L_2 = (X^+ L \cap L X^+)\setminus (X^+ L X^+ \cup L^2) =
(X^+ L_1 \cap L_1 X^+)\setminus (L_1 X^* L_1).
\end{array}
\]

\begin{corollary}
\label{regchain}

Let $A$ be a monomial algebra. We have that $A$ is an automaton algebra
if and only if the 1-chain language $L_1(A)$ is a regular language.
In this case, all chain languages $L_k(A)\ (k\geq 0)$ are also regular.
\end{corollary}

\begin{proof}
Put $L_k = L_k(A)$, for all $k$. For the characterization of automaton algebras,
note that $L = X^* L_1 X^*$ and $L_1 = L \setminus (X^+ L \cup L X^+)$ where
$L = L(I)$. Then, the closure properties of regular languages of Proposition \ref{closure}
imply that $L$ is regular if and only if $L_1$ is regular. Moreover, by the same
properties and by the formulas (\ref{govform}) we have that the powers $L^k$
and therefore the chain languages $L_k$ are also regular languages, for all $k$.
\end{proof}

In Section \ref{CSsec} we have explained that the rational Hilbert series $\HS(A)$
of an automaton algebra $A = F/I$ can be obtained by applying Algorithm \ref{orbit}
to the regular language $L(I) = X^* L_1 X^*$ where $L_1$ is a minimal monomial
basis of $I$. This is usually the most effective way to compute $\HS(A)$
as explained in \cite{LS2,LST}. Nevertheless, if $\gd(A) < \infty$ and one computes
by Algorithm \ref{orbit} the regular grammars $G_k$ of all regular chain languages
$L_k = L_k(A)$, the rational generating functions $\gamma(L_k)$ are obtained by
solving the linear systems $S(G_k)$ and the Hilbert series $\HS(A)$ is given
by the formula (\ref{hilbhom}).

In contrast with the regular case, it is easy to find a 1-chain language
$L_1$ which is an unambiguous cf-language but the corresponding 2-chain
language $L_2$ is not even context-free.

\begin{Example}
\label{countex}

Let $X = \{x,y,z\}$ and consider the minimal monomial basis
\[
L_1 = \{x^ny^nz\mid n\geq 2\} \cup \{xy^nz^n\mid n\geq 2\}.
\]
Since the two sets above are disjoint, one has immediately that
$L_1$ is an unambiguous cf-language which is generated by the linear
grammar $G_1 = (V,X,P,S)$ where $V = \{S, A, B\}$ and
\[
P = \{S\to A z\mid x B, A\to x^2y^2\mid x A y, B\to y^2z^2\mid y B z\}.
\]
We have therefore that $\gamma(L_1)$ is a rational function.
One can easily compute the 2-chain language $L_2 = \{x^ny^nz^n\mid n\geq 2\}$
which is not a context-free language because it does not satify the context-free
pumping lemma (see, for instance, \cite{HMU}). Nevertheless, the generating function
$\gamma(L_2)$ is clearly a rational one and since the global dimension is exactly 3,
we conclude that the corresponding Hilbert series
\[
H = \left( 1 - n t + \gamma(L_1) - \gamma(L_2) \right)^{-1}
\]
is a rational function.
\end{Example}


\section{Unambiguous monomial algebras}
\label{unsec}

The aim to construct monomial algebras having algebraic Hilbert series which
are computable by means of Theorem \ref{CSth}, motivate the following definition.

\begin{definition}
Let $A$ be a monomial algebra where $L_1 = L_1(A)\subset (X^+)^2$. We call
$A$ a {\em homologically unambiguous monomial algebra}, briefly an {\em unambiguous
algebra}, if all chain languages $L_i(A)\ (i\geq 1)$ are unambiguous cf-languages.
\end{definition}

By Corollary \ref{regchain}, it is clear that the class of unambiguous algebras
generalizes the class of automaton algebras.

\begin{proposition}
\label{ufingd}
Let $A$ be an unambiguous algebra having finite global dimension.
Then, the Hilbert series $\HS(A)$ is an algebraic function.
\end{proposition}

\begin{proof}
Put $k = \gd(A) - 1$ and denote by $L_i = L_i(A)$ ($1\leq i\leq k$).
By Theorem \ref{CSth}, we have that each $\gamma(L_i)$ is an algebraic function,
that is, an algebraic element over the rational function field $\KK(t)$.
By the formula (\ref{hilbhom}), we obtain hence that $\HS(A)$ is also
an algebraic element.
\end{proof}

Under the hypothesis of the above result, we may want to compute in practice
an algebraic equation satisfied by $\HS(A)$. To this purpose, we can proceed
in the following way. If $G_i$ is an unambiguous cf-grammar corresponding
to $L_i$, we construct the corresponding algebraic system $S(G_i)$ over
disjoint sets of commutative variables, for each index $i$. Then, let $E_i$
be the start variable of $G_i$ and consider the linear equation
\begin{equation}
\label{eceq}
E = 1 - n t - \sum_{i=1}^k (-1)^i E_i.
\end{equation}
Clearly, the variable $E$ corresponds to the Euler characteristic $1/\HS(A)$.
Denote by $S$ the algebraic system obtained by joining all the equations
in $S(G_i)$ ($1\leq i\leq k$) together with equation (\ref{eceq}). Note that
the coefficients of the commutative polynomials in $S$ are in the rational
function field $\KK(t)$. By eliminating in $S$ for the variable $E$, we obtain
a polynomial $p(E)$ with one root equal to $1/\HS(A)$. According to Theorem \ref{lexgb},
we may use to this purpose the computation of a \Gr\ basis where $E$ is
the lowest variable in some lexicographic monomial ordering. The required
polynomial for $\HS(A)$ is therefore $q(H) = H^d p(1/H)$ where $d = \deg(p)$.
We will illustrate this method by the examples contained in Section \ref{monsec}
and Section \ref{fpsec}.

We provide now general results which are useful to construct a large class
of unambiguous algebras. The following notations are assumed in all these results.
Let $X = \{x_1,\dots,x_n\}$ and $Y = \{y_1,\dots,y_m\}$ be two disjoint sets
of variables and put $Z = X\cup Y$. Then, we consider the free associative algebra
$F = \KK\langle Z\rangle$.

\begin{lemma}
\label{ulang}
Fix $k\geq 1$ and let $d_i\geq 0$, for any $1\leq i\leq k$. Then, let $R_0\subset X^*,
R_{ij}\subset X^+\, (1\leq j\leq d_i + 1)$ and $L_{ij}\subset Y^+\, (1\leq j\leq d_i)$
be unambiguous cf-languages, for all $i$. Moreover, assume that the set
\[
L = R_0\cup \bigcup_{1\leq i\leq k} R_{i1} L_{i1} R_{i2} \cdots
R_{id_i} L_{id_i} R_{i{d_i+1}}
\]
is a disjoint union. Then, $L$ is an unambiguous cf-language.
\end{lemma}

\begin{proof}
Let $S_0, S_{ij}, S'_{ij}$ be the start variables of the unambiguous context-free
grammars $G_0, G_{ij}, G'_{ij}$ corresponding to $R_0, R_{ij}, L_{ij}$, respectively.  
A cf-grammar $G$ generating $L$ is clearly obtained by adding to all productions
of the cf-grammars $G_0, G_{ij}, G'_{ij}$ the new rule
\[
S\to S_0\mid S_{11} S'_{11} S_{12} \cdots S_{1d_1} S'_{1d_1} S_{1{d_1+1}}\mid
\ldots\mid S_{k1} S'_{k1} S_{k2} \cdots S_{kd_k} S'_{kd_k} S_{k{d_k+1}}.
\]
We show that the cf-grammar $G$ is also unambiguous. Let $w\in L = L(G)$.
Since $L$ is a disjoint union, either $w\in R_0$ or
$w\in R_{i1} L_{i1} R_{i2} \cdots R_{id_i} L_{id_i} R_{i{d_i+1}}$ for a unique
index $i$. If $w\in L_0$ and $S_0\lsto w$ is the unique leftmost derivation
of $w$ in $G_0$, then $w$ has the following unique leftmost derivation in $G$
\[
S\to S_0\lsto w.
\]
Otherwise, if $w\in R_{i1} L_{i1} R_{i2} \cdots R_{id_i} L_{id_i} R_{i{d_i+1}}$
then one has the unique factorization
\[
w = w_1 w'_1 w_2 \cdots w_{d_i} w'_{d_i} w_{d_i+1}
\]
with $w_j\in R_{ij}\subset X^+$ and $w'_j\in L_{ij}\subset Y^+$. Let $S_{ij}\lsto w_j$ and
$S'_{ij}\lsto w'_j$ be unique leftmost derivations in $G_{ij}$ and $G'_{ij}$,
respectively. Then, the unique leftmost derivation of $w$ in $G$ is
\begin{equation*}
\begin{gathered}
S\to S_{i1} S'_{i1} S_{i2} \cdots S_{id_i} S'_{id_i} S_{id_i+1} \lsto 
w_1 S'_{i1} S_{i2} \cdots S_{id_i} S'_{id_i} S_{id_i+1} \lsto \\
w_1 w'_1 S_{i2} \cdots S_{id_i} S'_{id_i} S_{id_i+1} \lsto \ldots \lsto 
w_1 w'_1 w_2 \cdots w_{d_i} w'_{d_i} S_{id_i+1} \lsto w.
\end{gathered}
\end{equation*}
\end{proof}

\begin{theorem}
\label{uchain}

Let $L\subset Y^+$ be an unambiguous context-free language and let $R_0\subset X^*,
R_1, R'_1, \ldots, R_k, R'_k \subset X^+$ be regular languages such that
their disjoint union
\[
L'_1 = R_0 \cup R_1 \cup R'_1 \cup \cdots \cup R_k \cup R'_k
\]
is a minimal monomial basis. Consider the monomial algebra $A = F/I$, where
the monomial ideal $I$ is generated by 
\[
L_1 = R_0 \cup R_1 L R'_1 \cup \cdots\cup  R_k L R'_k.
\]
Then, $A$ is homologically unambiguous.
\end{theorem}

\begin{proof}
Since $L'_1$ is a minimal monomial basis, we have immediately that $L_1$ is also
a minimal monomial basis and hence $L_1(A) = L_1$.
Moreover, by the closure properties of cf-languages of Proposition \ref{closure},
one has that $L_1$ is a cf-language. Finally, by Lemma \ref{ulang}
we conclude that $L_1$ is unambiguous. 

For $q\geq 2$, consider the subset $S_q$ of $L_q = L_q(A)$ of the $q$-chains having
the form
\[
u l r'_i, r_i l v\ \mbox{and}\ r_i l w l' r_j,
\]
where $r_i\in R_i, r'_i\in R'_i\, (1\leq i\leq k), l,l'\in L$ and $u,v,w\in X^*$.
This subset $S_q$ can be described by formulas that are similar to (\ref{govform}).
Namely, if $M = L(I) = Z^* L_1 Z^*$ and $T = M\cap X^* = X^* R_0 X^*$, one has to replace
in (\ref{govform}) all occurrences of the powers $M^t$ with the set
\[
\sum_{i=1}^k\sum_{j=1}^t T^{j-1} R_i L R'_i T^{t-j}. 
\]
It follows that the set $S_q$ is the union of the languages
\[
C^i_q L R'_i, R_i L {C'}^i_q\ \mbox{and}\ R_i L Q^{ij}_q L R'_j,
\]
where the sets $C^i_q$ are defined as
\[
\begin{array}{rcl}
C_{2k}^i & = & (X^+ T^{k-1} R'_i \cap T^k X^+)\setminus (X^+ T^k X^+ \cup T^k R'_i); \\
C_{2k-1}^i & = & (X^+ T^{k-1} X^+ \cap T^{k-1} R'_i) \setminus
(X^+ T^{k-1} R'_i \cup T^k X^+),
\end{array}
\]
the sets ${C'}^i_q$ are defined by similar formulas where $R'_i$ is replaced by $R_i$ and
the sets $Q^{ij}_q$ are obtained as
\[
\begin{array}{lcl}
Q^{ij}_{2k} & = & (X^+ T^{k-1} R'_j \cap R_i T^{k-1} X^+)\setminus
(X^+ T^k X^+ \cup R_i T^{k-1} R'_j); \\
Q^{ij}_{2k-1} & = & (X^+ T^{k-1} X^+ \cap  R_i T^{k-2} R'_j) \setminus
(X^+ L^{k-1} R'_j \cup R_i T^{k-1} X^+).
\end{array}
\]
Observe that the languages $C^i_q, {C'}^i_q, Q^{ij}_q\, (1\leq i,j\leq k)$ are regular
owing to the closure properties of the regular languages. 

Consider now the automaton algebra $B = \KK\langle X\rangle/J$, where $J$ is the
monomial ideal generated by $R_0$. It is clear that $L_q = L_q(A)$ contains $L_q(B)$.
Moreover, a chain $w\in L_q\setminus L_q(B)$ either belongs to the languages
$C^i_q, {C'}^i_q, Q^{ij}_q$ or it contains some overlapping subwords belonging
to these sets in such a way that the whole $w$ is covered by these subwords. 
We conclude that the set $L_q$ is the disjoint union of the regular language $L_q(B)$
and the sets 
\[
C^{i_1}_{q_1} L Q^{i_1 i_2}_{q_2} L \cdots L Q^{i_{d-1} i_d}_{q_d} L C'^{i_d}_{q_{d+1}},
\]
where $q_1 + \cdots + q_{d+1} = q$. Now, the result follows from Lemma \ref{ulang}.
\end{proof}

According to Proposition \ref{ufingd}, an unambiguous algebra of the above class
having finite global dimension holds an algebraic Hilbert series. 
We conclude this section by showing that in the class of unambiguous
algebras of Theorem \ref{uchain}, there are algebras with infinite global
dimension which have also (computable) algebraic Hilbert series.

\begin{theorem}
\label{uchain2}

Let $L \subset Y^+$ be an unambiguous cf-language and let $R, R'\subset X^+$ be two
regular languages such that their disjoint union $L'_1 = R\cup R'$ is a minimal
monomial basis. Consider the monomial algebra $A = F/I$ where the monomial ideal $I$
is generated by $L_1 = R L R'$. Then, $A$ is an unambiguous algebra such that
$\gd(A) = \infty$ and $\HS(A)$ is an algebraic function which is rationally dependent
on the generating function $\gamma = \gamma(L)$, that is, $\HS(A)$ is an element of the
algebraic extension $\KK(t)(\gamma)$.
\end{theorem}

\begin{proof}
By Theorem \ref{uchain}, we have that $A$ is an unambiguous algebra where $L_1(A) = L_1$.
Denote $L_k = L_k(A)$ the $k$-chain language of $A$, for $k\geq 1$. By the definition
of 2-chains, we have that $L_2 = R L Q L R'$ where $Q$ is the language
of all minimal overlaps of the elements of $R$ with the elements of $R'$, that is
\[
Q = (R X^* \cap X^* R') \setminus R X^* R'.
\]
Note that $Q$ is a regular language. More generally, for $k\geq 2$, we have that
\[
L_k = R (L Q)^{k-1} L R'.
\]
All these languages are unambiguous by Lemma \ref{ulang}. By the homology of $A$,
we obtain that
\begin{equation*}
\begin{gathered}
\HS(A)^{-1} = 1 - (n + m) t - \sum_{k\geq 1}(-1)^k \gamma(L_k) = \\
1 - (n + m) t + \sum_{k\geq 1} \gamma(R) ( - \gamma(L) \gamma(Q) )^{k-1} \gamma(L) \gamma(R') = \\
1 - (n + m)t + \frac{\gamma(R) \gamma(R') \gamma(L)}{1 + \gamma(Q) \gamma(L)}.
\end{gathered}
\end{equation*}
Recall finally that the generating function $\gamma(R), \gamma(R'), \gamma(Q)$ are rational
functions because $R,R',Q$ are regular languages.
\end{proof}


\section{Monomial algebra examples}

\label{monsec}

We propose now a couple of examples of unambiguous monomial algebras
with finite global dimension to illustrate the proposed methods for computing
algebraic Hilbert series.

\begin{Example}

Fix $X = \{x,y,z,c\}$ and $Y = \{a,b\}$. We put $Z = X\cup Y$ and
$F = \KK\langle Z\rangle$. Consider the Lukasiewicz cf-grammar
$G = (V,Y,P,S)$ where $V = \{S\}$ and $P = \{S\to a \mid b S S\}$.
The corresponding cf-language $L = L(G)$ consists of the algebraic
expressions in Polish notation. For instance, these expressions
of length $\leq 7$ are the following ones
\[
a,
b a a,
b a b a a,
b b a a a,
b a b a b a a,
b a b b a a a,
b b a a b a a,
b b a b a a a,
b b b a a a a.
\]
It is well-known that $G$ and hence $L$ are unambiguous. Then, we consider
the language
\[
L_1 = \{x^2y, x^2z, xy^2, xyz, xzy, xz^2 \}\cup yz^2 L c.
\]
By Lemma \ref{ulang}, we have that $L_1$ is also an unambiguous cf-language
whose generating (unambiguous) cf-grammar is $G_1 = (V_1,Z,P_1,S)$ where
$V_1 = \{S,T\}$ and
\[
P_1 = \{S\to x^2y\mid x^2z\mid xy^2\mid xyz\mid xzy\mid xz^2\mid yz^2 T c,
T\to a \mid b T T\}.
\]
Moreover, the language $L_1$ is clearly a minimal monomial basis and we denote by
$I\subset F$ the two-sided ideal generated by $L_1$. Then, we consider the monomial
algebra $A = F/I$ for which we want to compute the Hilbert series $\HS(A)$ by means
of the formula (\ref{hilbhom}). We compute therefore the 2-chain language $L_2$
which is
\[
L_2 = \{x^2y^2, x^2yz, x^2zy, x^2z^2\} \cup \{xyz^2, xy^2z^2, xzyz^2\} L c.
\]
Finally, we can compute the 3-chain language $L_3$ as
\[
L_3 = \{x^2y^2z^2, x^2zyz^2\} L c.
\]
One has immediately that $L_4 = 0$, that is, $\gd(A) = 4$. In accordance with
Theorem \ref{uchain}, the sets $L_1,L_2,L_3$ are all unambiguous cf-languages, that is,
$A$ is an unambiguous algebra. Hence, we can apply Theorem \ref{CSth} to compute
the generating functions $\gamma(L_1),\gamma(L_2),\gamma(L_3)$. Denote by $E = 1/\HS(A)$
the inverse of the Hilbert series (Euler characteristic) of $A$ and let $E_i = \gamma(L_i)$
($1\leq i\leq 3$). By the formula (\ref{hilbhom}), we obtain the linear
equation $E = 1 - 6 t + E_1 - E_2 + E_3$. By applying Theorem \ref{CSth}
for the generating functions $E_i$, we obtain all together the algebraic system
\[
\left\{
\begin{array}{ccl}
E & = & 1 - 6 t + E_1 - E_2 + E_3, \\
E_1 & = & 6 t^3 + t^4 T, \\
E_2 & = & 4 t^4 + (t^5 + 2t^6) T, \\
E_3 & = & 2 t^7 T, \\
T & = & t + t T^2.
\end{array}
\right.
\]
Assume $\Char(\KK) = 0$. With respect to the lexicographic monomial ordering
of the polynomial algebra $R = \KK(t)[E,E_1,E_2,E_3,T]$ with
$E\prec E_1\prec E_2\prec E_3\prec T$, one computes the corresponding \Gr\ basis
\begin{equation*}
\begin{gathered}
E^2 + (- 2 t^6 + 2 t^5 + 9 t^4 - 13 t^3 + 12 t - 2) E +
(4 t^{14} - 8 t^{13} + 8 t^{11} - 11 t^{10} \\
+ 18 t^9 + 9 t^8 - 70 t^7 + 56 t^6 + 52 t^5 - 87 t^4 + 13 t^3 + 36 t^2 - 12 t + 1), \\
(2 t^3 - 2 t^2 - t + 1) E_1 - E + (- 12 t^6 + 12 t^5 + 2 t^4 - 6 t + 1), \\
(2 t^2 - 2 t - 1) E_2 + (2 t + 1) E_1 + (-2 t - 1) E + (- 8 t^6 - 12 t^2 - 4 t + 1), \\
E_3 - E_2 + E_1 - E + (- 6 t + 1), 2 t^7 T - E_3.
\end{gathered}
\end{equation*}
One easily obtains the roots of the first quadratic polynomial as
\[
1 - 6 t + \frac{13}{2} t^3 - \frac{9}{2} t^4 - t^5 + t^6
\pm t^3 (1 - t) (1 - 2 t^2) \frac{\sqrt{1 - 4 t^2}}{2}.
\]
We conclude that
\[
\HS(A) = 
\left(1 - 6 t + \frac{13}{2} t^3 - \frac{9}{2} t^4 - t^5 + t^6
- t^3 (1 - t) (1 - 2 t^2) \frac{\sqrt{1 - 4 t^2}}{2}\right)^{-1}
\]
since this function has a power series expansion with correct coefficients,
namely
\[
\HS(A) = 1 + 6 t + 36 t^2 + 210 t^3 + 1228 t^4 + 7175 t^5 + 41929 t^6 +
245017 t^7 + \ldots.
\]

\end{Example}


\begin{Example}

Let $X = \{x,y,z,c,d\}, Y = \{a,b\}$ and put $Z = X\cup Y,
F = \KK\langle Z\rangle$. Denote again by $L$ the Lukasiewicz language
on the alphabet $Y$ and define the set
\[
L_1 = c L \{x^2y, xyz, xzx\} \cup \{xy^2, y^2z, z^2y \} L d.
\]
From Lemma \ref{ulang} it follows that $L_1$ is an unambiguous cf-language
whose generating (unambiguous) cf-grammar $G_1$ has productions
\[
S\to c T x^2y\mid c T xyz\mid c T xzx\mid xy^2 T d, y^2z T d, z^2y T d,
T\to a \mid b T T.
\]
Observe that $L_1$ is a minimal monomial basis and denote $I$ the two-sided
ideal generated by $L_1$ and $A = F/I$ the corresponding monomial algebra.
We can easily compute the 2-chain language of $A$ as
\[
L_2 = c L \{x^2y^2, x^2y^2z, xyz^2y, xzxy^2\} L d
\]
which is generated by the unambiguous cf-grammar $G_2$ with productions
\[
S\to c T x^2y^2 T d\mid c T x^2y^2z T d\mid c T xyz^2y T d\mid c T xzxy^2 T d,
T\to a \mid b T T.
\]
Finally, it is clear that $L_3 = \emptyset$. In accordance with Theorem \ref{uchain},
we see that $A$ is an unambiguous algebra. Denote $E = 1/\HS(A)$ and put
$E_i = \gamma(L_i)$ ($1\leq i\leq 2$). Then, one has the linear equation
$E = 1 - 7 t + E_1 - E_2$. Now, by Theorem \ref{CSth} we obtain the following
algebraic system
\[
\left\{
\begin{array}{ccl}
E & = & 1 - 7 t + E_1 - E_2, \\
E_1 & = & 6 t^4 T, \\
E_2 & = & (t^6 + 3 t^7) T^2, \\
T & = & t + t T^2.
\end{array}
\right.
\]
Assuming $\Char(\KK) = 0$, we compute the corresponding \Gr\ basis
for the lexicographic monomial ordering of $R = \KK(t)[E,E_1,E_2,T]$
with $E\prec E_1\prec E_2\prec T$, namely
\begin{equation*}
\begin{gathered}
E^2 + (- 6 t^7 - 2 t^6 + 3 t^5 + t^4 - 6 t^3 + 14 t - 2) E +
(9 t^{14} + 6 t^{13} + t^{12} \\
- 18 t^{10} - 6 t^9 - 6 t^8 - 8 t^7 + 23 t^6 + 4 t^5 - 43 t^4 + 6 t^3 + 49 t^2 - 14 t + 1), \\
(3 t^2 + t - 6) E_1 + 6 E + (- 18 t^7 - 6 t^6 + 42 t - 6), \\
E_2 - E_1 + E + (7 t - 1), 6 t^4 T - E_1.
\end{gathered}
\end{equation*}
The Hilbert series of $A$ is the inverse of one of the roots of the first
quadratic polynomial, namely
{\small\[
\HS(A) = \left( 1 - 7 t + 3 t^3 - \frac{t^4}{2} - \frac{3 t^5}{2} + t^6 + 3 t^7 
- t^3 (6 - t - 3 t^2) \frac{\sqrt{(1 - 2 t)(1 + 2 t)}}{2} \right)^{-1}.
\]}
This is confirmed by its correct power series expansion
\[
\HS(A) = 1 + 7 t + 49 t^2 + 343 t^3 + 2401 t^4 + 16801 t^5 + 117565 t^6 +
822655 t^7 + \ldots.
\]

\end{Example}


\begin{Example}
We end this section with a simple example of an unambiguous algebra with infinite
global dimension but computable (algebraic) Hilbert series. Let $X = \{x\},
Y = \{a,b\}, Z = X\cup Y$ and $F = \KK\langle Z\rangle$. Consider the Dyck language
$L$ on the alphabet $Y$ and denote by $\gamma = \gamma(L)$ its generating function.
One computes easily that
\[
\gamma = \frac{1 - \sqrt{1 - 4 t^2}}{2 t^2}.
\]
Consider the unambiguous 1-chain language
\[
L_1 = x L x\subset Z^*.
\]
Let $I\subset F$ be the two-sided ideal generated by $L_1$ and put $A = F/I$.
For any $n\geq 1$, the (unambiguous) $n$-chain language of $A$ is clearly
\[
L_n = x ( L x )^n.
\]
We conclude that $\gd(A) = \infty$ and $\gamma(L_n) = t^{n+1} \gamma^n$.
In accordance with Theorem \ref{uchain2}, we finally obtain that
\[
\HS(A)^{-1} = 1 - 3 t + t^2\frac{\gamma}{(1 - t\gamma)} = 
\frac{1 - 6 t + 6 t^2 - (1 - 4 t)\sqrt{1 - 4 t^2}}{1 - 2 t - \sqrt{1 - 4 t^2}}.
\]
\end{Example}


\section{Finitely presented algebra examples}
\label{fpsec}

Finitely presented graded algebras are important structures in the noncommutative setting.
We now propose an example of such an algebra having an algebraic Hilbert series
which can be computed by the results and the methods proposed in this paper.
To find other examples of this kind, one has to go back to 1980-1981 \cite{Ko,Sh}
where the corresponding algebraic Hilbert series were computed by combinatorial
enumeration of normal words.

Fix $X = \{a',b',x,y\}, Y = \{a,b,e\}$ and put $Z = X\cup Y, F = \KK\langle Z\rangle$.
Consider the graded ideal $I\subset F$ which is generated by the following
(noncommutative) polynomials
\begin{itemize}

\item[(i)] $a'x - xa', b'x - xe$;

\item[(ii)] $a'a - aa', a'b - ab', b'a - ba', b'b - bb', a'e - ab, b'e - b^2$;

\item[(iii)] $ay - y^2, by - y^2, a'y - y^2, b'y - y^2$;

\item[(iv)] $xy$.

\end{itemize}
Let $A = F/I$ be the corresponding finitely presented graded algebra.
We want to prove that the Hilbert series $\HS(A)$ is a (non-rational) algebraic
function in a constructive way, that is, we want to compute explicitely a polynomial,
with coefficient in the rational function field $\KK(t)$, such that $\HS(A)$ is one
of its roots.

We assume now that the reader is familiar with the theory of noncommutative
\Gr\ bases which are also called \Gr-Shirshov bases. For a complete
reference see, for instance, \cite{BC,Mo,Uf2}. We recall here some basic notations
and results. We fix a monomial ordering of the free associative algebra
$F = \KK\langle Z\rangle$. Then, let $0\neq f = \sum_{i=1}^k c_i w_i\in F$
with $0\neq c_i\in \KK, w_i\in X^*$ and $w_1\succ w_2\succ \ldots w_k$.
The word $\lm(f) = w_1$ is called the {\em leading monomial of $f$}.
A (possibly infinite) subset $U\subset I$ is called a {\em \Gr-Shirshov basis},
briefly a {\em GS-basis of $I$}, if $\lm(U) = \{\lm(f)\mid 0\neq f\in U\}\subset Z^*$
is a monomial basis of the monomial ideal
\[
\LM(I) = \langle \lm(f)\mid 0\neq f\in I \rangle\subset F.
\]
The GS-basis $U$ is said {\em minimal} if the monomial basis $\lm(U)$ is such.
We call $\LM(I)$ the {\em leading monomial ideal of $I$}. If $J = \LM(I)$,
one defines the corresponding monomial algebra $B = F/J$. It is easy to prove
that $\HS(A) = \HS(B)$. Then, for the given algebra $A$, our aim is to prove
that $B$ is a (non-automaton) unambiguous algebra with finite global dimension
in order to apply Theorem \ref{CSth} to the chain languages of $B$ and
compute $\HS(B)$.

It is useful to consider the cf-language $L\subset Y^*$ which is defined by
the cf-grammar $G = (V,Y,P,S)$ where $V = \{S,T\}$ and
\[
P = \{S\to 1\mid T e S, T\to 1\mid a T b T\}.
\]
Recall that the production $T\to 1\mid a T b T$ defines the Dyck language
of all words of balanced brackets $a,b$. We denote this language by $D$.
Clearly, we have that $L = (D e)^*$. Since $D$ is unambiguous, by similar
arguments to the ones in Lemma \ref{ulang}, one obtains that $L$ is also
an unambiguous cf-language.

We fix the graded (left) lexicographic monomial ordering on $F$ with
$a'\succ b'\succ a\succ b\succ e\succ x\succ y$. If $B = F/J$ with $J = \LM(I)$,
observe that $L_1(B) = \lm(U)$ where $U$ is a minimal GS-basis of $I$.

\begin{theorem}
\label{fpex}
With the notations above, we have that $L_1(B)$ is the disjoint union of the
finite set of the leading monomials of the binomials (i)-(iii) with the
unambiguous cf-language $x L y$. It follows that $\gd(A) = \gd(B) = 3$, where
\[
L_2(B) = \{a',b'\} \{a,b\} y\cup \{a',b'\} x L y.
\]
Then, the monomial algebra $B$ is homologically unambiguous and the Hilbert
series $\HS(A) = \HS(B)$ is an algebraic function. Precisely, the Euler
characteristic $1/\HS(A)$ satisfies the quadratic equation
{\small \[
E^2 + (2 t - 1)(5 t^2 - 10 t + 2) E 
+ (2 t - 1)(13 t^5 - 56 t^4 + 85 t^3 - 50 t^2 + 12 t - 1) = 0
\]}
and we have that
\[
\HS(A) = \left( 1 - 7 t + \frac{25}{2} t^2 - 5 t^3 +
t^2\frac{\sqrt{1-4t^2}}{2} \right)^{-1}.
\]
\end{theorem}

\begin{proof}
We start by computing formally a (non-minimal) GS-basis $U'$
of $I$. Note immediately that the leading monomials of the binomials (i)-(iii)
are their leftmost words which implies that such binomials form themselves
a (finite) minimal GS-basis. We show that all other elements of $U'$
are words of type $x u y^m$, for some $u = u(a,b,e)\in Y^* = \{a,b,e\}^*$
and $m\geq 1$. We argue by induction on the (countable) steps of the
critical-pair-and-completion algorithm for computing GS-bases
(see, for instance, \cite{Mo}). 


The basis for the induction consists in the word (iv). Moreover,
at any step of the algorithm, the only new overlappings appear between
the words $x u y^m$ and the leading monomials of (i). This leads to words
of the form $x w(a',b,e) y^m$ where $w(a',b,e)\in \{a',b,e\}^*$. Then,
reductions by (ii) and (iii) eliminate all prime signs in the word $w$,
and finally, reductions by (iii) lead to a word of the desired form $x u(a,b,e) y^m$.
These elements are not reducible by the binomials in $U'$, that is, they
are new elements of $U'$ of the desired form.

Denote now by $M$ the set of all words $x u y^m$ which are irreducible
with respect to all other elements of $U'$. In other words, the minimal
GS-basis $U$ in the union of the binomials (i)-(iii) with $M$.
By considering overlappings, each word $w\in M$ is obtained
as a complete reduction of one of the elements $a'q$ or $b'q$,
where $q$ is another element of $M$. Moreover, both complete reductions
of $a'q$ or $b'q$ are either zero or belong to $M$, for each $q\in M$.
Since $q$ is also obtained by a complete reduction of either $a'q'$
or $b'q'$ for some $q'\in M$, we conclude that $w$ is a complete reduction
of either $a'^2q', a'b'q', b'a'q'$ or $b'^2q'$. Again, for each $q'\in M$,
all such non-zero reductions belong to $M$. By continuing in the same way,
we conclude that the set $M$ is exactly the set of complete reductions
of the words of the form $s x y$, where $s = s(a',b')$ runs over all words
in the letters $a',b'$.

To describe the elements of $M$, let us consider then any word $s = s(a',b')$.
The reductions of $s x y$ via (i) lead to a word $x s' y$, where
$s' = s(a',e)$ (that is, the same word as $s$ where $b'$ is replaced by $e$).
Then, the reductions via (ii) force moving prime signs toward the right,
up to the first appearance of one the letters $y$ or $e$. In the first case,
the letters $a',b'$ jump over $y$ toward the right and then they are
replaced by $y$. In the second case, the prime sign is eliminated with
replacing $e$ by $b$, so that this new $b$ form a balanced pair with
the last unbalanced $a$ or $a'$ before it. Then, the reduction process
leads to a word $x u(a,b,e) y^m$ with no prime signs, where the word $u$
is ended by $e$.

We claim that the word $u$ should have the form $u_1 e u_2 e \cdots u_k e$
where all $u_i = u_i(a,b)$ are balanced subwords, that is, $u_i\in D$ and hence
$u\in L = (D e)^*$. Indeed, each $u_i$ is obtained from some $w_i = w_i(a',b')$
via the reductions by (ii). Then, $w_i$ contains the same amount of $a'$-s
as of $b'$-s since all reductions by (ii) save this equality.

By contradiction, assume now that $w_i$ is not balanced and then some initial
subword of $w_i$ contains less amount of $a'$-s then of $b'$-s. Therefore,
the corresponding prefix $u'$ of $u_i$ contains less amount of prime signs then
the one of $e$-s. Note that each reduction by (ii) cannot decrease the difference
between the number of prime signs and the number of $e$-s in a prefix of any word.
It follows that some letter $e$ in $u'$ will never be involved in the reduction
process. This contradicts the assumption that $u_i$ does not contain $e$.
Thus, $w_i$ is balanced and hence $u\in L$.

Finally, for every $u\in L$ we show that the word $x u y$ belongs to $M$. Indeed,
$x u y$ is a complete reduction of the word $w x y$, where $w$ is the image of $u$
under the substitutions $a \mapsto a'$ and $b,e \mapsto b'$.
It follows that the words of the form $x u y^m$ ($u\in L$) cannot be reduced 
by other elements of $U'$ if and only if $m = 1$. We conclude that $M = x L y$
and the minimal GS-basis $U$ has the desired form.


From the description above of the minimal GS-basis $U\subset I$ and hence
of the 1-chain language $L_1(B) = \lm(U)$, the given formula for $L_2(B)$
follows immediately and we conclude that $\gd(B) = 3$. Then, $B$ is an unambiguous algebra
in accordance with Theorem \ref{uchain}. It is clear how to obtain the unambiguous
cf-grammars generating the chain languages of $B$ and therefore we obtain the following
algebraic system
\[
\left\{
\begin{array}{ccl}
E & = & 1 - 7 t + E_1 - E_2, \\
E_1 & = & 12 t^2 + t^2 S, \\
E_2 & = & 4 t^3 + 2 t^3 S, \\
S & = & t S T + 1, \\
T & = & t^2 T^2 + 1. \\
\end{array}
\right.
\]
Assuming $\Char(\KK) = 0$, we compute the corresponding lexicographic \Gr\ basis
with $E\prec E_1\prec E_2\prec S\prec T$, which is
\begin{equation*}
\begin{gathered}
E^2 + (10 t^3 - 25 t^2 + 14 t - 2) E +
(26 t^6 - 125 t^5 + 226 t^4 \\
- 185 t^3 + 74 t^2 - 14 t + 1),
(2 t - 1) E_1 + E + (- 20 t^3 + 7 t - 1), \\
E_2 - E_1 + E + (7 t - 1),
2 t^3 S - E_2 + 4 t^3, \\
(2 t^4 - t^3) T + (- 2 t^3 + 3 t^2 - t) S + E + (6 t^3 - 15 t^2 + 8 t - 1).
\end{gathered}
\end{equation*}
The roots of the first (univariate) quadratic polynomial are the following
algebraic functions
\[
1 - 7 t + \frac{25}{2} t^2 - 5 t^3 \pm t^2\frac{\sqrt{1 - 4 t^2}}{2}.
\]
The Hilbert series $\HS(B)$ is inverse of one of them, namely
\[
\HS(B) = \left(1 - 7 t + \frac{25}{2} t^2 - 5 t^3 +
t^2\frac{\sqrt{1 - 4 t^2}}{2}\right)^{-1}.
\]
In fact, this function admits the following power series expansion having
the correct coefficients
\[
\HS(B) =
1 + 7 t + 36 t^2  + 166 t^3  + 730 t^4  + 3139 t^5  + 13350 t^6  + 56466 t^7  + \ldots.
\]

Finally, by the Anick resolution \cite{An} one has that $\gd(A)\leq \gd(B) = 3$.
Because $\HS(A)$ is a non-rational function, we conclude that $\gd(A) = 3$.
\end{proof}

Note that the above algebra is a quadratic one. For two quadratic algebras $A$ and $B$
with unambiguous associated monomial algebras, one shows easily that their tensor product
$A\otimes B$, free product $A*B$ and Manin's black dot product $A\bullet B$ have also
unambiguous associated monomial algebras (with respect to suitable monomial orderings).
This gives other examples of finitely presented algebras with algebraic Hilbert series.

Moreover, observe that one can modify the given example to obtain other finitely
presented algebras with (non-rational) algebraic Hilbert series. Consider, for instance,
the algebra $A'$ with generators $a',b',a,b,c,d,e,x,y$ and relations (i),(ii) together
with a modified version of (iii),(iv), namely
\begin{itemize}

\item[(iii)$'$] $a y - y c, by - y d, a'y, b'y$;

\item[(iv)$'$] $x y e$.

\end{itemize}

With respect to the graded lexicographic monomial ordering on the free associative
algebra with $a'\succ b'\succ a\succ b\succ c\succ d\succ e\succ x\succ y$,
denote $B'$ the leading monomial algebra corresponding to $A'$. By similar arguments
to the ones of Theorem \ref{fpex}, one has that $L_1(B')$ is the disjoint union
of the finite set of the leading monomials of the relations (i)-(iii) with the
unambiguous cf-language $x L y D' e$, where $L = (D e)^*$ and $D,D'$ are the Dyck
languages on the alphabets $\{a,b\},\{c,d\}$, respectively. It follows that $B'$
is an unambiguous algebra, $\gd(B') =\gd(A') = 3$ and
\[
L_2(B') = \{ a',b'\} \{ a,b\} y \cup \{a',b' \} x L y D' e.
\]
The corresponding Hilbert series $\HS(B') = \HS(A')$ is the following algebraic
function
\[
\left(1 - 9 t + \frac{23}{2} t^2 - 3 t^3 + t^2\frac{\sqrt{1 - 4 t^2}}{2}\right)^{-1}.
\]
We conclude by recalling that examples of finitely presented quadratic algebras
with intermediate growth and hence transcendental Hilbert series are provided in 
\cite{Kc1,Kc2}.


\section{Conclusions and future directions}

\label{conc}

In this paper, we have shown how to construct and compute with classes
of noncommutative monomial algebras whose Hilbert series are
(non-rational) algebraic functions. Examples of finitely presented graded
algebras whose corresponding leading monomial algebras belong to
the proposed class (and hence possess algebraic Hilbert series)
have also been given. We believe that our results contribute to fill up
the lack of knowledge about algebraic Hilbert series that was in fact
present in the previous literature.

In doing this, we have also established connections between apparently
far worlds as the homology of noncommutative algebras and the theory
of formal grammars of theoretical computer science. Using a similar
interplay between algebra, combinatorics and computer science,
we suggest to study other classes of algebras with more general homologies
from the perspective of their Hilbert series which are possibly
D-finite functions. 

\section{Acknowledgements}

The authors want to express their gratitude to the \singular team for
allowing access to their computational resources in Kaiserslautern.
We would also like to thank Alexandr Podoplelov for sharing his C++ code
of the computer algebra system \bergman \cite{Ba2,CPU}. The author D.P.
is grateful to the University of Bari for its creative atmosphere and
warm hospitality. Finally, we would like to thank the anonymous reviewers
for the careful reading of the manuscript. We have sincerely appreciated
all valuable comments and suggestions as they have improved the readability
of the paper.

\end{document}